\documentclass[a4paper]{amsart}
\pdfoutput=1
\usepackage{amsthm}
\usepackage{amssymb}
\usepackage{bm}
\usepackage{bbding}
\usepackage{biblatex}
	\bibliography{bibliography.bib}

\theoremstyle{definition}
\newtheorem{defn}{Definition}
\newtheorem*{nota}{Notational Convention}

\theoremstyle{plain}
\newtheorem{lem}{Lemma}

\newtheorem{cor}{Corrolary}

\title{Another Approach on Power Sums}

\begin{document}
\author[C. Muschielok]{Christoph Muschielok}
\address{Department of Chemistry, Technical University of Munich, 85748 Garching, Germany}
\email{c.muschielok(at)tum.de}

\begin{abstract}
	We show that explicit forms for certain polynomials~$\psi^{(a)}_m(n)$
	with the property
	\[ \psi^{(a+1)}_m(n) = \sum_{\nu=1}^n \psi_m^{(a)}(\nu) \] 
	can be found (here, $a,m,n\in\mathbb{N}_0$). We use these polynomials
	as a basis to express the monomials~$n^m$. Once the expansion
	coefficients are determined, we can express the $m$-th power
	sums~$S^{(a)}_m(n)$ of any order $a$,
	\[ S^{(a)}_m(n) = \sum_{\nu_a = 1}^n \cdots \sum_{\nu_2 = 1}^{\nu_3} \sum_{\nu_1=1}^{\nu_2} \nu_1^m\,, \]
	in a very convenient way by exploiting the summation property of the
	$\psi_m^{(a)}$,
	\[ S^{(a)}_m(n) = \sum_k c_{mk} \psi_k^{(a)}(n)\,. \]
\end{abstract}
\maketitle

\section{Introduction}
Power sums have been studied for a long time, for example by Nicomachus of
Gerasa, who demonstrated that the sum of cubes is the square of a triangular
number, that is
\begin{align}
	S_3^{(1)}(n) = \sum_{\nu = 1}^n \nu^3 = \left[ \frac{n(n+1)}{2} \right]^2\,.
\end{align}
This was known to Faulhaber who worked on power sums in the 17th century. In
his work \emph{Academia Algebr\ae} he presents formulas for the sums of odd
powers up to the 17th.\cite{knuth1993faulhaberpowersums} For odd powers, the
power sum can be represented by a polynomial of the triangular numbers. This
was shown by \citeauthor{jacobi1834powersums} shortly over 200 years
later.\cite{jacobi1834powersums} Jacob Bernoulli found a closed form for power
sums in 1713 (\emph{Summ{\ae} potestatum}) which can be written
\begin{align}
	S_m^{(1)}(n) = \sum_{\nu = 1}^n \nu^m = \frac{1}{m+1} \sum_{k=0}^m {m+1\choose k} B_k n^{m-k+1}\,,
\end{align}
with the binomial coefficient~$m+1 \choose k$ and the Bernoulli numbers~$B_k$
($B_1 = 1/2$). This work was published posthumously in
1713.\cite{bernoulli1713arsconjectandi} 

In the following, we define polynomials~$\psi_m(n)$ which are related to some
pyramidal numbers and use them to express a power~$n^m$. The coefficients of
this expansion can be used together with a generalization of the~$\psi_m(n)$
which we also give here, to yield the value of any power sum.

\section{Definition and Properties of the $\psi_m$ Polynomials}
\begin{defn}
        We define the polynomial sequences
        \begin{align}
    	    \psi_m(n) = n + (m-1) (n-1) B_{m-1, n-1}\,,
        \end{align}
        where $m,n\in\mathbb{N}$ are natural numbers and~\[B_{a,b} =
        \frac{(a+b)!}{a!\,b!}\] is the binomial coefficient.
\end{defn}
The~$\psi_m(n)$ yield the sequences of the $m$-th powers of~$n$ for $m<4$.
That is, the quadratic and cubic numbers, $n^2$ and $n^3$, for $m=2$ and $m=3$,
and the ``linear numbers'' $n$ for $m=1$. This is of no particular interest, we
could probably find some families of sequences, for which this is true. To see,
why this could indeed be interesting, we need to introduce further sequences.
\begin{defn}
	On top of the $\psi_m$, we recursively define the polynomial
	sequences~$\psi^{(a)}_m$ for $a \in \mathbb{N}_0$,
        \begin{align}
    	    \psi_m^{(a)}(n) = \psi_m^{(a)}(n-1) + \psi^{(a-1)}_m(n)\,.
		\label{eq:RecursivePsiMA}
        \end{align}
        Notationwise, we use the convention~$\psi_m^{(0)} \equiv \psi_m$.
\end{defn}
\begin{lem}
	\label{lem:PsiMASum}
	It immediately follows from Eq.~\eqref{eq:RecursivePsiMA}, that
	$\psi_m^{(a)}(n)$ is given by the sum of all~$\psi_m^{(a-1)}(\nu)$ for
	$1 \nu \le $:
        \begin{align}
		\psi^{(a)}_m = \sum_{\nu=1}^n \psi_m^{(a-1)}(\nu)\,.
		\label{eq:PsiMASum}
        \end{align}
	The $\psi_m^{(a)}(n)$ are the polynomial series of $a$-th order with
	respect to the sequence~$\psi_m(n)$.
\end{lem}
\begin{proof}[Proof of Lemma~\ref{lem:PsiMASum}]
	We use Eq.~\eqref{eq:RecursivePsiMA} time after time to rewrite the
	$\psi_m^{(a)}$ term until we are left with Eq.~\eqref{eq:PsiMASum}:
	\begin{align*}
		\psi_m^{(a)}(n) &= \psi_m^{(a)}(n-1) + \psi_m^{(a-1)}(n)\,,\\
				&= \psi_m^{(a)}(n-2) + \psi_m^{(a-1)}(n-1) + \psi_m^{(a-1)}(n)\,,\\
				&\dotsc\\
				&= \psi_m^{(a)}(n-k) + \sum_{\nu=n-k+1}^{n} \psi_m^{(a-1)}(\nu)\,,\\
		\psi_m^{(a)}(n) &= \sum_{\nu=1}^n \psi_m^{(a-1)}(\nu)\,. \qedhere
	\end{align*}
\end{proof}

We can state a lemma about linear combinations of functions which show
the recursive property of Eq.~\eqref{eq:RecursivePsiMA}.
\begin{lem}
	\label{lem:RecursiveLinearCombinations}
	The recursive property Eq.~\eqref{eq:RecursivePsiMA} and its
	reformulation as the sum in Eq.~\eqref{eq:PsiMASum} hold for any
	sequence which is a certain type of linear combination of the
	$\psi_m^{(a)}$. Let the coefficients $c_1, c_2 \in \mathbb{K}$ be
	elements of some field.
	\begin{align}
		f^{(a)}(n) = c_1 \psi_m^{(a)}(n) + c_2 \psi_{m'}^{(a)}(n) 
		\implies f^{(a)}(n) = f^{(a)}(n-1) + f^{(a-1)}(n)\,.
		\label{eq:RecursiveLinearCombinations}
	\end{align}
\end{lem}
\begin{proof}[Proof of Lemma~\ref{lem:RecursiveLinearCombinations}]
	Insert the recursive property into the $\psi$-terms of
	Eq.~\eqref{eq:RecursiveLinearCombinations} and evaluate:
	\begin{align}
		f^{(a)}(n) &= c_1 \left[ \psi_m^{(a)}(n-1) + \psi_m^{(a-1)}(n) \right] 
		              + c_2 \left[ \psi_{m'}^{(a)}(n-1) + \psi_{m'}^{(a-1)}(n) \right]\,,\\
			   &= \underbrace{c_1 \psi_m^{(a)}(n-1) + c_2 \psi_{m'}^{(a)}(n-1)}_{f^{(a)}(n-1)}
			      + \underbrace{c_1 \psi_m^{(a-1)}(n) + c_2 \psi_m^{(a-1)}(n)}_{f^{(a-1)}(n)}\,,\\
		f^{(a)}(n) &= f^{(a)}(n-1) + f^{(a-1)}(n)\,. \qedhere
	\end{align}
\end{proof}

\section{Linking the $\psi_m$ Polynomials to Power Sums}
Why bother with all of this? Notice, how the~$\psi_m(n)$ are just the $m$-th
powers of $n$. If we can find a closed form for~$\psi_m^{(a)}(n)$, we
automatically have the $a$-th power sum
\begin{align}
	S_m^{(a)}(n) &= \sum_{\nu_a=1}^{n} \sum_{\nu_{a-1}=1}^{\nu_a} \cdots \sum_{\nu_1=1}^{\nu_2} \nu^m\,.
	\label{eq:PowerSumA}
\end{align}
However, this holds only for $m<4$. Fortunately, we can take the nice property of
Eq.~\eqref{eq:RecursivePsiMA} or Eq.~\eqref{eq:PsiMASum} with us.
\begin{nota}
	For multiple summations with common ultimate summation boundaries in
	which the intermediate upper summation boundary of each sum is given by
	the index of the next sum, as in Eq.~\eqref{eq:PowerSumA}, we want to
	introduce the notation
	\[ \sum_{\boldsymbol{\nu}_a = 1}^{n} \nu^m\,,\]
	where the multi-index notation~$\boldsymbol{\nu}_a$ means $(\nu_a, \nu_{a-1}, \dotsc, \nu_1)$.
\end{nota}
Therefore, we want to find an expansion of the power~$n^m$ in terms of the $\psi_m(n)$,
\begin{align}
	n^m = \sum_{k} c_{mk} \psi_k(n)\,,
\end{align}
so that we can rewrite the power sums of $a$-th order as
\begin{align}
	\sum_{\boldsymbol{\nu}_a=1}^{n} \nu^m &= \sum_{\boldsymbol{\nu}=1}^n \sum_k c_{mk} \psi_m(\nu)\,,\\
					      &= \sum_k c_{mk} \sum_{\boldsymbol{\nu}=1}^n \psi_m(\nu)\,,\\
					      &= \sum_k c_{mk} \psi_m^{(a)}(n)\,.
\end{align}
Notice, how all the $a$ sums over the~$\nu_i$ are swallowed by the basis
functions~$\psi_k$ by multiple use of Eq.~\eqref{eq:PsiMASum} and turn them
into~$\psi_k^{(a)}$.

The important bit is: if we have a closed form for the~$\psi_m^{(a)}(n)$, once
we know the set of $c_mk$, we have the value for any~$S_m^{(a)}(n)$. In the
following, we first show that indeed one can find such a closed form and that
we can easily obtain the values of the expansion coefficients.

\section{Finding a Closed Form for the Series Polynomials}
In the following, it is our goal to find an expression for
the~$\psi_m^{(a)}(n)$. It turns out, that a good starting point for this is to
realize, that a similar identity to~Eq.~\eqref{eq:RecursivePsiMA}, holds for
the binomial coefficient:
\begin{align}
	B_{a,b} = B_{a-1,b} + B_{b, a-1}\,.
	\label{eq:RecursiveBinomialCoeff}
\end{align}
This is just what we see in Pascal's triangle and we can show this by a few
simple steps of algebra, after inserting the definition for each symbol. With
the same argument with which we proved Lemma~\ref{lem:PsiMASum}, we may write
\begin{align}
	B_{a,b} = \sum_{\beta=1}^b B_{a-1,\beta}\,.
	\label{eq:BinomialCoeffSum}
\end{align}

Due to this identity, we can already find a closed form for
the~$\psi_m^{(a)}(n)$ in terms of binomial coefficients as we show in the rest
of this section. Let us first rewrite the expression for~$\psi_m(n)$ as a
linear combination of binomial coefficients.
\begin{align}
	\psi_m^{(a)}(n) &= n + (m-1) (n-1) B_{m-1, n-1}\,,\\
			&= B_{1, n-1} + m (m-1) B_{m, n-2}\,.
\end{align}
Cancelling the factor $n-1$ from the binomial coefficient~$B_{m-1, n-1}$, so
that we can make a~$B_{m, n-2}$ out of the second summand, leads to the term
formally not being defined for $n=1$. We have to make sure, that the limit of
the series at $n=1$ has still a defined value.
\begin{align}
	\lim_{n=1} B_{m, n-2} &= \lim_{n=1} \frac{(n+m-2)!}{m!\,(n-2)!}\,,\nonumber\\
	                      &= \lim_{n=1} (n-1) \frac{(n+m-2)!}{m!\,(n-1)!}\,,\nonumber\\
	\lim_{n=1} B_{m, n-2} &= 0\,.
\end{align}
\begin{lem}
	\label{lem:ClosedFormPsiMA}
	The elements of the series of $a$-th order $\psi_m^{(a)}(n)$ have the closed form
	\begin{align}
		\psi_m^{(a)}(n) = B_{a+1,n-1} + \frac{m(m-1)}{m+a} (n-1) B_{m+a-1,n-1}\,.
		\label{eq:ClosedFormPsiMA}
	\end{align}
\end{lem}
\begin{proof}[Proof of Lemma~\ref{lem:ClosedFormPsiMA}]
	We will use a proof via induction and begin from
	Eq.~\eqref{eq:PsiMASum} for $a=1$ and insert the binomial coefficient
	representation for~$\psi_m(n)$.
	\begin{align}
		\psi_m^{(1)}(n) &= \sum_{\nu=1}^n \psi_m(n)\,,\\
		                &= \sum_{\nu=1}^n \left[ B_{1,\nu-1} + m(m-1) B_{m, \nu-2} \right]\,,\\
				&= \sum_{\nu=1}^n B_{1,\nu-1} + m(m-1) \sum_{\nu=1}^n B_{m, \nu-2}\,,\\
				&= B_{2,n-1} + m(m-1) B_{m+1, n-2}\,,\\
				&= B_{2,n-1} + \frac{m(m-1)}{m+1} (n-1) B_{m, n-1}\,.
	\end{align}
	This is just the form given by~Eq.~\eqref{eq:ClosedFormPsiMA} for
	$a=1$, so that we have a valid start for the induction. Now suppose,
	that Eq.~\eqref{eq:ClosedFormPsiMA} holds for any~$a$. We have to prove,
	that it holds also for~$a+1$:
	\begin{align}
		\psi_m^{(a+1)}(n) &= \sum_{\nu=1}^n \psi_m^{(a)}(\nu)\,,\\
		                  &= \sum_{\nu=1}^n \left[ B_{a+1,\nu-1} + \frac{m(m-1)}{m+a} (\nu-1) B_{m+a-1,\nu-1} \right]\,,\\
				  &= \sum_{\nu=1}^n B_{a+1,\nu-1} + \frac{m(m-1)}{m+a} \sum_{\nu=1}^n (\nu-1) B_{m+a-1,\nu-1}\,,\\
				  &= B_{a+2,n-1} + m(m-1) \sum_{\nu=1}^n B_{m-a,\nu-2}\,,\\
				  &= B_{a+2,n-1} + m(m-1) B_{m+a+1,n-2}\,,\\
				  &= B_{a+2,n-1} + \frac{m(m-1)}{m+a+1} (n-1) B_{m+a,n-1}\,.
	\end{align}
	This is just Eq.~\eqref{eq:ClosedFormPsiMA} for $a+1$ substituted
	for~$a$. Thus, the~$\psi_m^{(a)}(n)$ have indeed the proposed closed
	form.
\end{proof}

\section{Coefficients of the Monomial Expansion of $\psi_m$}
With the expression for the $\psi_m^{(a)}(n)$ ready, what remains to do is to
find the coefficients~$c_{mk}$. Before we tackle this problem, let us write
the~$\psi_m(n)$ in terms of powers of~$n$, at first, that is we expand it in
terms of the monomials
\begin{align}
	\psi_m(n) = \sum_{i=0}^m a_{mi} n^i\,.
	\label{eq:amiExpansion}
\end{align}
We can find the coefficients~$a_{mi}$ in Eq.~\eqref{eq:amiExpansion} using
Vieta's formulas.\cite{bronstein2008taschenbuch} For this, we
rewrite~$\psi_m(n)$ as
\begin{align}
	&\psi_m(n) = n + \frac{1}{(m-2)!} f_m(n)\,,\\
	\intertext{with $f_m(n)$ given by}
	&f_m(n) = \prod_{k=-1}^{m-2} (n + k) = \prod_{i=1}^m (n - \alpha_i) = \sum_{k=0}^m b_{m,m-k} n^k\,,\\
	\intertext{where the $\alpha_i$ are of course the integer roots of this
	polynomial:}
	&\alpha_i = -(i - 2), 1 \le i \le m\,.
\end{align}
The coefficients~$b_{m,m-k}$ are connected to our coefficients of interest~$a_{mi}$ by
\begin{align}
	a_{m,m-k} = \frac{b_{m,m-k}}{(m-2)!} + \delta_{1k},\,,m \ge 2\,,
\end{align}
where we use the Kronecker symbol~$\delta_{1k}$ to account for the additional
term~$n$ of~$\psi_m$ with respect to $f_m$. We discuss the cases $m=1$ and
$m=0$ later. Until then, we consider everything under the condition~$m\ge2$.
\begin{cor}
	By inserting the roots $\alpha_i$ into Vieta's formulas, it can be
	verified that the expansion coefficients of $f_m(n)$ for the lowest orders are given
	by the following expressions:
	\begin{align}
		b_{m,1}   &= -(m-2)!\,,\\
		b_{m,0}   &= 0\,.
	\end{align}
	Together with $b_{mm}=1$, we find for the actual expansion coefficients $a_{mk}$ of~$\psi_m$
	\begin{align}
		a_{m0}    &= 0\,,\\
		a_{m1}    &= 0\,,\\
		a_{mm}    &= \frac{1}{(m-2)!}\,.
	\end{align}
\end{cor}
As we want to consider only classical polynomials, we write, subsuming the
results for the~$b_{ml}$, for the expansion coefficient~$a_{mk}$
\begin{align}
	a_{mk} = 0,\text{ if }  k < 2 \lor k > m,\,.
\end{align}
Therefore, the non-zero values for the $a_{mk}$ are those with~$2 \le k \le m$.

We turn now to the remaining cases $m = 0$ and $m = 1$. For the latter, we
already mentioned, that $\psi_1(n) = n$. Thus, its single expansion coefficient
is~$a_{11} = 1$. Power sums of $a$-th order of $n$ are given by the $n$-th
simplicial $a$-polytopic number\cite{oeis2022simplexnumbers}
\begin{align}
	\sigma_{a}(n) = \frac{(n+a-1)!}{(n-1)!\,a!} = B_{a, n-1}\,.
\end{align}
This is qualitatively different to the case of the~$\psi_m$ with $m \ge 2$. For
those, we have in general a linear combination of two binomial coefficients,
whereas for $m=1$ we can express it also as a single binomial coefficient. 

The case $m = 0$ is not included in how we defined the $\psi_m$ here. However,
it can be reduced to $m=1$: clearly, $n^0 = 1$, so that any power sum of $a$-th
order can be reduced to a power sum of $n$ of $(a-1)$-th order. In the
following, we will restrict ourselves to $m \ge 2$.

\section{Recursive Definition of the $\psi_m$-Expansion Coefficients of $n^m$}
Expressing $\psi_m$ in terms of the monomials~$\lbrace{n^k}\rbrace_{k=2}^m$ and
expressing~$n^m$ in terms of the $\lbrace{\psi_k}\rbrace_{k=2}^m$, are
transformations between a pair of dual bases. Thus, the expansion coefficients
must build mutually inverse square matrices~$A_m = (a_{\mu\kappa})_{2\le \mu,\kappa \le m}$ 
and $C_m = (c_{\mu\kappa})_{2\le \mu,\kappa \le m}$, such that
\begin{align}
	\sum_{\kappa=2}^m a_{\mu\kappa} c_{\kappa\mu'} = \delta_{\mu\mu'}, (2\le \mu,\mu' \le m)\,.
	\label{eq:InverseMatrices}
\end{align}
Then, we can solve this for~$c_{\mu\mu'}$, the first term in the sum in
Eq.~\eqref{eq:InverseMatrices} which is non-zero, to obtain a recursive
expression for these coefficients.
\begin{align}
	c_{\mu\mu'} &= \frac{1}{a_{\mu\mu}} \left( \delta_{\mu\mu'} - \sum_{\kappa=2}^{\mu-1} a_{\mu\kappa} c_{\kappa\mu'}\right)\,,\\
		    &= (\mu-2)! \left( \delta_{\mu\mu'} - \sum_{\kappa=2}^{\mu-1} a_{\mu\kappa} c_{\kappa\mu'} \right)\,.
\end{align}
For clarity, we truncated the upper summation boundary to explicitly include
only non-zero values, $\mu' \le \kappa$, for $c_{\kappa\mu'}$. Alternatively,
we can write
\begin{align}
	c_{ml} &= \begin{cases} 
			1/a_{mm} = (m-2)!                       & \text{if } l = m\,,\\
			- (m-2)! \sum_{k=2}^{m-1} a_{mk} c_{kl} & \text{if } 2 \le l \le m-1\,,\\
			0                                       & \text{else.}
	          \end{cases}
\end{align}
This solves our problem: we now have a closed expression for the
$\psi^{(a)}_m(n)$ and also the transformation coefficients~$c_{mk}$. 
Of course, we can also build the matrix~$A_m$ and calculate its inverse.

As an example, we give the matrix $C_8 = (c_{\mu\kappa})_{2\le \mu,\kappa\le m}$ 
with explicit values up to $\mu=8, \kappa=8$. This matrix includes the $C_\mu$
matrices, $\mu < 8$, as square submatrices, which are obtained by truncating
$C_8$ at the appropriate row and column:
\begin{align}
	C_8 = (c_{\mu\kappa})_{2 \le \mu, \kappa \le 8} 
	    = \begin{pmatrix} 
		1      &        &        &        &        &        &         \\
		0      & 1      &        &        &        &        &         \\
		1      & -2     & 2      &        &        & \mbox{$(0)_{\mu<\kappa}$}       &         \\
		0      & 5      & -10    & 6      &        &        &         \\
		1      & -10    & 40     & -54    & 24     &        &         \\
		0      & 21     & -140   & 336    & -336   & 120    &         \\
		1      & -42    & 462    & -1764  & 3024   & -2400  & 720     \\
	      \end{pmatrix}
\end{align}
The coefficients~$c_{m\kappa}$ within each row seem to be of alternating sign,
where the highest-order non-zero coefficient~$c_{mm}$ on the diagonal always
has positive sign. Furthermore, the coefficient~$c_{m2}$ apparently is~$0$ for
odd orders~$m$ and $1$ for even orders. Looking at the distribution of the
absolute values of the coefficients~$|c_{mk}|$ for a set order~$m$, it seems as
if it assumes a maximum value for some $k<m$.

\section{Various Power Sums}
With this, we can write down any power sum~$S_m^{(a)}(n)$. Exemplarily, we want to give 
expressions for some of the better known of them. It is easy for $m=2$ and $m=3$:
\begin{align}
	& S^{(1)}_2(n) = \sum_{\nu=1}^n \nu^2 = \psi^{(1)}_2(n) = \left[ 1 + \frac{2}{3} (n - 1) \right] B_{2, n-1} = \frac{1}{6} n (n+1) (2n+1)\,.\\
	& S^{(1)}_3(n) = \sum_{\nu=1}^n \nu^3 = \psi^{(1)}_3(n) = B_{2,n-1} + \frac{3}{2} (n - 1) B_{3,n-1} = \left[\frac{n(n+1)}{2}\right]^2\,.
\end{align}
We reproduce Nicomachus's formula for $m=3$. For $m=2$ we find the square
pyramidal numbers\cite{oeis2022squarepyramidalnumbers}. The expressions become
more difficult starting with $m=4$:
\begin{align}
	S^{(1)}_4(n) &= \sum_{\nu=1}^n \nu^4 = 2\psi^{(1)}_4(n) - 2\psi^{(1)}_3(n) + \psi^{(1)}_2(n)\,,\\
		     &= B_{2,n-1} + 4 (n-1) B_{4,n-1} -3 (n-1) B_{3,n-1} + \frac{2}{3} (n-1) B_{2,n-1} \,,\\
		     &= 4 (n-1) B_{4,n-1} -3 (n-1) B_{3,n-1} + \frac{2n + 1}{3} B_{2,n-1}\,,\\
	             &= \frac{1}{30} n \left( 6n^4 + 15n^3 +10n^2 - 1 \right) \\
	S^{(1)}_5(n) &= 5\psi_3^{(1)}(n) -10\psi_4^{(1)}(n) +6\psi_5^{(1)}(n)\,,\\
		     &= B_{2,n-1} + 20(n-1) B_{5,n-1} - 24 (n-1) B_{4,n-1} + \frac{15}{2} (n-1) B_{3,n-1}\,,\\
		     &= \frac{1}{12} n^2 (n+1)^2 [2n^2 + 2n - 1]\,.
\end{align}
As a final example, we furthermore give the expression for $S_8^{(2)}(n)$:
\begin{align}
	\begin{split}
		S^{(2)}_8(n) &= 720 \psi^{(2)}_8 - 2400 \psi^{(2)}_7 + 3024 \psi^{(2)}_6 \\
			     &\hphantom{=}-1764\psi^{(2)}_5 + 462\psi^{(2)}_4 -42\psi^{(2)}_3 + \psi^{(2)}_2\,,
	\end{split}\\
	\begin{split}
		\hphantom{S^{(2)_8(n)}} &= (n-1)\left( 4032 B_{9,n-1} - 11200 B_{8,n-1} + 11340 B_{7,n-1}\vphantom{\frac{252}{5}} \right.\\
					&\hphantom{=}\left. -5040 B_{6,n-1} +924 B_{5,n-1} -\frac{252}{5} B_{4,n-1} \right) \\
					&\hphantom{=}+  \left[ 1 + \frac{1}{2}(n-1) \right] B_{3,n-1}
	\end{split}\\
	\hphantom{S^{(2)}_8(n)} &= \frac{1}{180} n (n+1)^2 (n+2) (2n^2 + 4n -1) (n^4 + 4n^3 + n^2 -6n + 3)\,.
\end{align}

\section{Final Considerations}
Finally, we want to show how the expansion coefficients~$c_{mk}$ are related to the Bernoulli numbers~$B_{i}$, as part of our results is equivalent to Bernoulli's ($a=1$). To put this into context, we expand the~$\psi_k^{(1)}(n)$ and set equal to Bernoulli's formula:
\begin{align}
	\sum_{k=2}^m c_{mk} \psi_k^{(1)}(n) &= \frac{1}{m+1} 
					       \sum_{l=1}^{m+1} {m+1\choose l} B_{m-l+1} n^l\,,\\
	\sum_{k=2}^m c_{mk} \sum_{j=1}^{k+1} g_{kj} n^j  &= \frac{1}{m+1} 
						            \sum_{l=1}^{m+1} 
                                                              {m+1\choose l} B_{m-l+1} n^l\,,\\
	\sum_{j=1}^{m+1} 
          \left\lbrace \sum_{k=2}^m c_{mk} g_{kj} \right\rbrace
          n^j 
            &= \frac{1}{m+1} 
               \sum_{l=1}^{m+1} 
                 {m+1\choose l} B_{m-l+1} n^l\,.
\end{align}
We require that the expansion coefficients of~$\psi^{(1)}_k(n)$, $g_{kj} = 0$
if $j>k$. Similar to what we already have used before, we then can decouple the
upper summation boundary and switch the summation order. We may identify the
product of coefficients with the coefficients in Bernoulli's expansion:
\begin{align}
	\sum_{k=2}^m c_{mk} g_{kj} = \frac{1}{m+1} {m+1\choose j} B_{m-j+1}\,,
\end{align}
which we can solve for the Bernoulli number~$B_{m-j+1}$:
\begin{align}
	B_{m-j+1} = \frac{(m-j+1)!\,j!}{m!} \sum_{k=2}^m c_{mk} g_{kj}\,.
\end{align}
We can thus express the Bernoulli numbers in terms of the $c_{mk}$ and the
coefficients~$g_{kj}$ of the~$\psi_k^{(1)}$. 

In the end, we want to state explicitly the main advantage of expressing
the~$S_m^{(a)}(n)$ in terms of the $\psi^{(a)}$ instead of using a single
polynomial at once: instead of putting the complexity of the summation
procedure into the coefficients, we shift it into the basis functions. The
recursive trait of those basis functions~$\psi^{(a)}_k(n)$ then makes for an elegant
generalization for more complex sums ($a>1$).
\printbibliography
\end{document}